\documentclass[a4paper,10pt]{article}

\usepackage[T1]{fontenc}
\usepackage[utf8]{inputenc}
\usepackage[margin=1.9cm]{geometry}
\usepackage{amsmath,amssymb,amsthm,mathtools}

\usepackage{booktabs}
\usepackage{array}
\usepackage{graphicx}
\usepackage{xcolor}
\usepackage{hyperref}
\usepackage{cleveref}
\usepackage{setspace}
\usepackage{enumitem}
\usepackage{thmtools}
\usepackage{mdframed}
\usepackage{tikz}
\usetikzlibrary{arrows.meta,calc,decorations.pathmorphing,decorations.markings,shadings,patterns,positioning}
\usepackage{pgfplots}
\pgfplotsset{compat=1.18}
\usepackage{caption}
\usepackage{subcaption}
\usepackage{titlesec}
\usepackage{fancyhdr}

\usepackage[ruled,vlined]{algorithm2e}

\SetKwInput{KwParams}{Parameters}
\SetKwInput{KwInit}{initialize}
\SetKwInput{KwOutput}{Output}
\SetKw{KwStop}{stop}

\SetAlgoNlRelativeSize{-1}
\SetAlFnt{\small}
\SetAlCapFnt{\small}
\SetAlCapNameFnt{\small}
\SetAlgoCaptionSeparator{:}
\DontPrintSemicolon

\definecolor{darkblue}{RGB}{31,56,100}
\definecolor{medblue}{RGB}{46,80,144}
\definecolor{boxbg}{RGB}{238,244,251}
\definecolor{darkgrey}{RGB}{60,60,60}
\definecolor{keepgreen}{RGB}{20,120,60}
\definecolor{trimorange}{RGB}{190,100,10}
\definecolor{manifoldcol}{RGB}{221,233,246}
\definecolor{tancol}{RGB}{255,244,214}
\definecolor{tanline}{RGB}{193,68,14}

\newcommand{\manifoldpatch}{%
	\fill[manifoldcol]
	(-2.5,-0.55) .. controls (-1.6,0.30) and (-0.6,0.20)  .. (0.1,0.15)
	.. controls (1.0,0.10)  and (1.7,0.65)  .. (2.5,-0.05)
	-- (2.5,-0.45)
	.. controls (1.7,0.25)  and (1.0,-0.30) .. (0.1,-0.25)
	.. controls (-0.6,-0.20) and (-1.6,-0.10) .. (-2.5,-0.95)
	-- cycle;
	\draw[thick,darkblue]
	(-2.5,-0.55) .. controls (-1.6,0.30) and (-0.6,0.20) .. (0.1,0.15)
	.. controls (1.0,0.10)  and (1.7,0.65) .. (2.5,-0.05);
}

\hypersetup{
	colorlinks=true, linkcolor=medblue, citecolor=medblue, urlcolor=medblue,
	pdftitle={RIAG-R: Riemannian Inertial Adaptive Gradient with Restart}
}

\titleformat{\section}{\large\bfseries\color{darkblue}}{\thesection.}{0.6em}{}
[\vspace{-0.3em}\textcolor{medblue}{\rule{\linewidth}{0.5pt}}]
\titleformat{\subsection}{\normalsize\bfseries\color{medblue}}{\thesubsection.}{0.5em}{}

\mdfdefinestyle{thmstyle}{
	linecolor=medblue, linewidth=1pt, backgroundcolor=boxbg,
	innertopmargin=5pt, innerbottommargin=5pt, innerleftmargin=10pt,
	innerrightmargin=8pt, leftline=true, rightline=false,
	topline=false, bottomline=false
}
\mdfdefinestyle{defstyle}{
	linecolor=darkblue, linewidth=1pt, backgroundcolor=boxbg!60,
	innertopmargin=5pt, innerbottommargin=5pt, innerleftmargin=10pt,
	innerrightmargin=8pt, leftline=true, rightline=false,
	topline=false, bottomline=false
}
\mdfdefinestyle{remarkstyle}{
	linecolor=darkgrey!50, linewidth=0.8pt, backgroundcolor=white,
	innertopmargin=4pt, innerbottommargin=4pt, innerleftmargin=10pt,
	innerrightmargin=8pt, leftline=true, rightline=false,
	topline=false, bottomline=false
}

\newtheoremstyle{riagthm}{4pt}{4pt}{\itshape}{}{\bfseries\color{darkblue}}{.}{0.5em}{}
\newtheoremstyle{riagdef}{4pt}{4pt}{\normalfont}{}{\bfseries\color{medblue}}{.}{0.5em}{}
\newtheoremstyle{riagrm}{3pt}{3pt}{\normalfont}{}{\itshape}{.}{0.5em}{}

\makeatletter
\theoremstyle{riagthm}
\@ifundefined{c@theorem}{\newtheorem{theorem}{Theorem}[section]}{}
\@ifundefined{c@lemma}{\newtheorem{lemma}[theorem]{Lemma}}{}
\@ifundefined{c@corollary}{}{}
\@ifundefined{c@proposition}{\newtheorem{proposition}[theorem]{Proposition}}{}

\theoremstyle{riagdef}
\@ifundefined{c@definition}{\newtheorem{definition}[theorem]{Definition}}{}
\@ifundefined{c@assumption}{\newtheorem{assumption}{Assumption}}{}

\theoremstyle{riagrm}
\@ifundefined{c@remark}{\newtheorem{remark}[theorem]{Remark}}{}
\@ifundefined{c@example}{}{}
\makeatother

\DeclareMathOperator{\grad}{grad}

\newcommand{\M}{\mathcal{M}}
\newcommand{\norm}[1]{\left\lVert#1\right\rVert}
\newcommand{\inner}[2]{\langle #1,\, #2 \rangle}
\newcommand{\R}{\mathbb{R}}
\newcommand{\expmap}{\mathrm{exp}}
\newcommand{\logmap}{\mathrm{log}}

\newcommand{\defeq}{\coloneqq}

\begin{document}
	\sloppy
	
	\begin{center}
		{\color{darkblue}\rule{\linewidth}{2pt}}\\[0.4cm]
		{\LARGE\bfseries\color{darkblue}
			A Riemannian Inertial Adaptive Gradient Method with
			Momentum Restart}\\[0.3cm]
		{\color{medblue}\rule{\linewidth}{1pt}}\\[0.5cm]
		{\large Lateef O. Jolaoso$^1$\quad Professer V. Ndlovu$^2$ \quad Maggie Aphane$^{3}$}\\[0.3cm]
		{\normalsize
			$^1$Big Data Technologies and Innovation Lab,
			University of Hertfordshire, Hatfield, United Kingdom\\
			$^{1,2,3}$Department of Mathematics and Applied Mathematics,
			Sefako Makgatho Health Sciences University, Ga-Rankuwa, Pretoria, South Africa\\[0.2em]
			\texttt{L.Jolaoso@Herts.ac.uk}}\\[0.5cm]
		{\small\today}
	\end{center}

	\begin{abstract}
\noindent
We propose \emph{RIAG-R}, a Riemannian inertial adaptive-gradient method with gradient-triggered restart for smooth optimization on Riemannian manifolds. RIAG-R combines geodesic inertial extrapolation with an AdaGrad-type scalar step-size and an alignment test that suppresses momentum whenever the inertial direction has an uphill component. This provides adaptive step-size selection while safeguarding the iteration against poorly aligned momentum. We establish a one-step descent inequality and, under suitable step-size conditions, prove convergence to stationarity together with an $\mathcal{O}(\varepsilon^{-2})$ iteration complexity for finding an $\varepsilon$-stationary point. Under a Riemannian Polyak-\L ojasiewicz condition, we further obtain a linear convergence rate for the iterates generated by the algorithm. We then performed some numerical experiments on eight real-data instances across sphere, Stiefel, symmetric positive-definite, and Grassmann manifolds which show that restart remains inactive when inertial is well aligned, while becoming effective on the saddle-rich sphere problems.


\end{abstract}
	
	\section{Introduction}\label{Sec1:Intro}

Optimization problems with intrinsic geometric constraints arise throughout machine learning, data analysis, signal processing, and scientific computing. Typical examples include eigenvalue and principal-subspace problems on spheres and Grassmann manifolds, orthogonality-constrained models on Stiefel manifolds, and covariance-type problems on the cone of symmetric positive-definite matrices. Riemannian optimization treats such constraints as part of the geometry of the search space rather than as external Euclidean constraints; see, for example, \cite{absil2008,boumal2023}. This viewpoint permits first-order algorithms to move along feasible directions in tangent spaces and return to the manifold through an exponential map or a retraction.

The basic representative is Riemannian gradient descent (RGD). Given $x_n\in\M$, it replaces the Euclidean update $x_{n+1}=x_n-\alpha\nabla f(x_n)$ by
\[
 x_{n+1}=\expmap_{x_n}\!\bigl(-\alpha_n\grad f(x_n)\bigr).
\]
For smooth objectives, RGD inherits much of the familiar first-order complexity theory: under appropriate Lipschitz-gradient assumptions, one obtains global stationarity guarantees of order $O(\varepsilon^{-2})$ in the nonconvex setting \cite{boumal2019,zhang2016first}. Nevertheless, the basic method leaves two practical issues unresolved. First, a useful global smoothness constant is often unavailable or conservative, making fixed step-size selection difficult. Second, gradient descent can be slow in poorly conditioned regions, motivating the use of momentum or inertial extrapolation.

Both remedies are well developed in Euclidean optimization but become more delicate on manifolds. Adaptive methods such as AdaGrad alter the step-size using information accumulated from previous gradients \cite{duchi2011}; their Riemannian counterparts must respect a metric and tangent spaces that vary with the iterate \cite{becigneul2019,kasai2019}. Momentum creates an additional geometric difficulty: vectors attached to different points belong to different tangent spaces and cannot be combined without an appropriate geometric operation. Riemannian acceleration and momentum methods therefore rely on logarithmic maps, exponential maps, vector transports, or related constructions \cite{alimisis2020,ahn2020,alimisis2021}. In fact, curvature introduces distortion absent from the Euclidean setting, and a successful accelerated construction must explicitly control that distortion. In particular, Alimisis et al. develop a continuous-time view of Riemannian acceleration, whereas Ahn and Sra identify metric distortion as a central obstacle in extending Nesterov-type acceleration globally \cite{alimisis2020,ahn2020}.

A second difficulty is that momentum is not uniformly beneficial. When the extrapolation direction becomes poorly aligned with the current descent geometry, inertial can produce oscillation, overshooting, or prolonged stagnation. In Euclidean accelerated methods this observation motivates adaptive restart strategies, which discard momentum when a simple progress or alignment test indicates that it is no longer useful \cite{odonoghue2015,su2016differential}. The same idea is especially natural on a manifold, where the local geometry itself changes along the trajectory. This suggests the following question:
\begin{quote}
\emph{Can one combine adaptive step-size selection, Riemannian inertial, and a computationally inexpensive restart rule in a single first-order scheme while retaining transparent stationarity guarantees?}
\end{quote}

This paper answers this question through RIAG-R (Riemannian Inertial Adaptive Gradient with Restart). Starting from two consecutive iterates, RIAG-R forms the tangent-space direction $d_n=-\logmap_{x_n}(x_{n-1})$, scales it by an adaptive inertial coefficient, and extrapolates to $w_n$. Before accepting this extrapolation, the method checks the sign of $\inner{\grad f(x_n)}{d_n}$. A positive value means that the inherited momentum has an uphill component at the current point; RIAG-R then restarts by setting the inertial coefficient to zero. The gradient step is taken from $w_n$ with the scalar adaptive step-size $\alpha_n=\eta/(\sqrt{G_n}+\varepsilon)$, where $G_n$ accumulates squared Riemannian gradient norms. Thus the restart test changes only the momentum component and costs one additional inner product.

An important feature of this construction is that restart is not intended to operate at every iteration. Rather, it acts only when the local geometry and the inherited direction indicate that momentum may be detrimental. This distinction is borne out by the numerical experiments: on six of the eight tested configurations the restart condition is never activated, so RIAG-R reduces exactly to its restart-free inertial counterpart, whereas restart becomes active on the two saddle-rich sphere problems. Thus, the experiments provide a direct way of separating the benefit of inertial from the additional safeguarding effect of restart.

\subsection{Relation to existing Riemannian first-order methods}
The proposed method lies at the intersection of three established lines of work. Classical Riemannian gradient, conjugate-gradient, trust-region, and stochastic methods provide the geometric algorithmic framework \cite{absil2008,boumal2019,bonnabel2013,sato2016dai,zhu2017riemannian}. Riemannian adaptive methods extend Euclidean learning-rate mechanisms to curved spaces \cite{becigneul2019,kasai2019}. Riemannian acceleration and momentum methods investigate how inertial information can be represented and controlled when tangent spaces change from point to point \cite{alimisis2020,ahn2020,alimisis2021}. RIAG-R is not intended as a replacement for the more elaborate accelerated constructions in these works. Its purpose is different: it gives a simple gradient-based scheme in which adaptivity, inertial, and restart are visible as separate algorithmic components and can therefore be analysed and tested individually.

\subsection{Main contributions}
The main contributions of this work are summarized as follows.
\begin{enumerate}[leftmargin=2em,itemsep=4pt]
\item \textbf{An adaptive inertial Riemannian method with restart.}
We introduce RIAG-R, which combines geodesic inertial extrapolation, an AdaGrad-type scalar step-size, and a gradient-alignment restart mechanism. The restart suppresses momentum whenever $\inner{\grad f(x_n)}{d_n}>0$, thereby preventing the algorithm from accepting an inertial direction with a local uphill component. The resulting safeguard requires only one additional inner product per iteration.

\item \textbf{Global stationarity and complexity guarantees.}
Under suitable step-size stability conditions, we prove square-summability of the inter-iterate displacements and weighted summability of the Riemannian gradient norms. Consequently, the method admits subsequential stationarity; when the adaptive step-sizes are bounded away from zero, the full gradient sequence converges to zero, every cluster point is critical, and an $\varepsilon$-stationary point is obtained within $O(\varepsilon^{-2})$ iterations.

\item \textbf{Linear convergence under a Riemannian PL condition.}
Under a Riemannian Polyak--\L ojasiewicz condition, we establish geometric decay of the objective gap up to an explicitly quantified inertia-dependent residual. The result identifies the price introduced by extrapolation and clarifies the potential role of restart in controlling poorly aligned momentum.

\item \textbf{Numerical assessment across four manifold geometries.}
We evaluate RIAG-R on eight real-data problem instances involving the sphere, Stiefel, symmetric positive-definite, and Grassmann manifolds, and compare it with a non-inertial adaptive Riemannian gradient method and the corresponding restart-free inertial scheme. On the six Stiefel, SPD, and Grassmann configurations, restart is never triggered and RIAG-R coincides with the restart-free inertial method, whereas on both saddle-rich sphere instances restart becomes active and improves upon unsafeguarded inertial. On the larger Digits sphere problem, RIAG-R is the only method to attain a gradient norm below tolerance level within the allowed number of iterations. The experiments therefore isolate when inertial alone is sufficient and when the restart safeguard becomes useful.
\end{enumerate}

The remainder of the paper is organized as follows. Section~\ref{sec:prelim} introduces the Riemannian geometric notation and analytical tools used throughout the paper. Section~\ref{sec:algorithm} presents RIAG-R, explains the roles of adaptive inertial, gradient-triggered restart, and adaptive step-size selection, and Section~\ref{sec:convergence} develops the convergence analysis, including stationarity, iteration complexity, criticality of cluster points, and PL-type geometric behaviour. Section~\ref{sec:numerical} evaluates the method on eight real-data instances spanning four different manifold geometries and examines separately the effects of inertial and restart relative to non-inertial and restart-free baselines. Section~\ref{sec:conclusion} summarizes the theoretical and computational findings and discusses directions for further research.

	\section{Riemannian Geometry Preliminaries}
	\label{sec:prelim}
	
	We collect below the definitions and results used throughout the paper.
	We follow the conventions of \cite{docarmo1992} and \cite{absil2008}.
	
	\begin{definition}
		A \emph{Riemannian manifold} $(\M,\inner{\cdot}{\cdot})$ is a smooth
		manifold $\M$ equipped with a smoothly varying family of inner products
		$\inner{\cdot}{\cdot}_x$ on the tangent spaces $T_x\M$.
		The Riemannian norm is $\norm{v}_x=\langle v,v\rangle_{x}^{1/2}$.
		The \emph{Riemannian distance} between $p,q\in\M$ is
		\[
		d(p,q) \;=\; \inf\bigl\{\ell(\gamma) :\, \gamma \text{ smooth curve from }
		p \text{ to } q\bigr\},
		\]
		where $\ell(\gamma)=\int_a^b\norm{\dot\gamma(t)}_{\gamma(t)}\,dt$.
	\end{definition}
	
	\begin{definition}
		The \emph{Levi-Civita connection} $\nabla$ is the unique torsion-free,
		metric-compatible connection on $(\M,\inner{\cdot}{\cdot})$.
		For a curve $\gamma:[a,b]\to\M$ and $v\in T_{\gamma(a)}\M$, the
		\emph{parallel transport} $\mathcal{P}_{\gamma,a,t}: T_{\gamma(a)}\M\to
		T_{\gamma(t)}\M$ is the unique linear isometry satisfying
		$\nabla_{\dot\gamma(t)} V(t)=0$, $V(a)=v$.
		A smooth curve $\gamma$ satisfying $\nabla_{\dot\gamma}\dot\gamma=0$ is
		called a \emph{geodesic}.
	\end{definition}
	
	\begin{definition}
		For $x\in\M$ and $v\in T_x\M$, the \emph{exponential map}
		$\expmap_x(v)\in\M$ is defined as $\gamma(1)$, where $\gamma$ is the
		unique geodesic with $\gamma(0)=x$ and $\dot\gamma(0)=v$.
		$\M$ is \emph{geodesically complete} if $\expmap_x$ is defined on all of
		$T_x\M$.
		The \emph{logarithmic map} $\logmap_x:\M\to T_x\M$ is the local inverse of
		$\expmap_x$, satisfying $\expmap_x(\logmap_x(y))=y$ and
		$\norm{\logmap_x(y)}_x = d(x,y)$.
	\end{definition}
	
	\begin{definition}[Riemannian gradient]
		The \emph{Riemannian gradient} of $f:\M\to\R$ at $x$ is the unique
		vector $\grad f(x)\in T_x\M$ satisfying
		$\inner{\grad f(x)}{v}_x = df_x(v)$ for all $v\in T_x\M$.
		For embedded submanifolds of $\R^N$ with the induced metric,
		$\grad f(x)=\mathrm{Proj}_{T_x\M}(\nabla f(x))$.
	\end{definition}
	
	\begin{definition}
		\label{def:lipschitz}
		The gradient field $\grad f$ is \emph{$L$-Lipschitz continuous} if for any
		$x,y\in\M$ joined by a minimising geodesic $\gamma$:
		\[
		\norm{\mathcal{P}_{\gamma,0,1}\grad f(x) - \grad f(y)}_y
		\;\leq\; L\,d(x,y).
\]
\end{definition}

The fundamental tool in most convergence proofs is the following descent
inequality on manifolds.
\begin{lemma}[Riemannian Descent Lemma {\cite[Lemma~1]{boumal2019}}]
	\label{lem:descent-lemma}
	Let $f:\M\to\R$ have $L$-Lipschitz Riemannian gradient.
	Then for every $x\in\M$ and $v\in T_x\M$,
	\begin{equation}
		f\bigl(\expmap_x(v)\bigr)
		\;\leq\;
		f(x) + \inner{\grad f(x)}{v}_x + \frac{L}{2}\norm{v}_x^2.
		\label{eq:descent-lemma}
	\end{equation}
\end{lemma}

	We also record a useful elementary inequality.
	
	\begin{lemma}[Young's inequality]
		\label{lem:young}
		For all $a,b>0$ and $\tau>0$,
		\begin{equation}
			ab \;\leq\; \frac{\tau}{2}a^2 + \frac{1}{2\tau}b^2,
			\qquad\text{in particular } ab\leq\tfrac12(a^2+b^2).
			\label{eq:young-scalar}
		\end{equation}
		Equivalently, for all $a,b\in\R$ and $\theta>0$,
		\begin{equation}
			(a+b)^2 \;\leq\; (1+\theta)a^2 + \Bigl(1+\frac1\theta\Bigr)b^2.
			\label{eq:young-square}
		\end{equation}
	\end{lemma}
	
	\section{The RIAG-R Algorithm}
	\label{sec:algorithm}
	
	We consider the optimization problem
	\begin{equation}
		\min_{x\in\M}\; f(x),
		\label{eq:problem}
	\end{equation}
	where $f:\M\to\R$ is continuously differentiable and $\M$ is a complete
	Riemannian manifold.
	We work throughout under the following standing assumptions.
	
	\begin{assumption}[$L$-Lipschitz gradient]
		\label{ass:lipschitz}
		The Riemannian gradient $\grad f$ is $L$-Lipschitz continuous on $\M$
		(Definition~\ref{def:lipschitz}).
	\end{assumption}
	
	\begin{assumption}[Lower bound]
		\label{ass:lower}
		There exists $f_*\in\R$ such that $f(x)\geq f_*$ for all $x\in\M$.
	\end{assumption}

	\subsection{Algorithm Description}

	\begin{algorithm}[H]
		\caption{RIAG-R: Riemannian Inertial Adaptive Gradient with Restart}
		\label{alg:RIAG-R}
		
		\KwParams{initial points $x_0,x_1\in\M$; 
			$\delta,\eta>0$, $\mu\in(0,1)$, $\varepsilon>0$}
		\KwInit{$G_0\leftarrow0$}
		
		\For{$n=1,2,3,\ldots$}{
			\eIf{$x_n\neq x_{n-1}$}{
				$\bar{\mu}_n\leftarrow
				\min\{\mu,\mu\,d(x_n,x_{n-1})\}$
			}{
				$\bar{\mu}_n\leftarrow\mu$
			}
			
			$d_n\leftarrow-\logmap_{x_n}(x_{n-1})$
			
			\eIf{$\inner{\grad f(x_n)}{d_n}>0$}{
				$\tilde{\mu}_n\leftarrow0$
			}{
				$\tilde{\mu}_n\leftarrow\bar{\mu}_n$
			}
			
			$w_n\leftarrow\expmap_{x_n}(\tilde{\mu}_n d_n)$
			
			$G_n\leftarrow G_{n-1}+\norm{\grad f(w_n)}^2$
			
			$\alpha_n\leftarrow\eta/(\sqrt{G_n}+\varepsilon)$
			
			$x_{n+1}\leftarrow
			\expmap_{w_n}(-\alpha_n\grad f(w_n))$
			
			\If{$\norm{\grad f(w_n)}<\delta$
				\textbf{ or } $d(x_{n+1},x_n)<\delta$}{
				\KwStop
			}
		}
		
		\KwOutput{sequence $\{x_n\}$}
		
	\end{algorithm}
	
	\begin{remark}[Interpretation of the three safeguards]
\begin{enumerate}[label=(\roman*),leftmargin=2em,itemsep=3pt]
\item \textbf{Adaptive inertial.} The coefficient $\bar\mu_n=\min\{\mu,\mu d(x_n,x_{n-1})\}$ automatically damps the extrapolation as consecutive iterates approach one another. In particular, $d(w_n,x_n)=\tilde\mu_n d(x_n,x_{n-1})\leq \mu d(x_n,x_{n-1})^2$, a second-order displacement estimate used later in the convergence analysis.
\item \textbf{Gradient-based restart.} The sign of $\inner{\grad f(x_n)}{d_n}_{x_n}$ tests whether the inherited momentum is locally aligned with ascent. If it is positive, RIAG-R sets $\tilde\mu_n=0$, hence $w_n=x_n$. This is the key safeguard that removes the first-order uphill contribution of the inertial step from the descent estimate.
\item \textbf{Adaptive gradient step.} Since $G_n$ is nondecreasing, $\alpha_n=\eta/(\sqrt{G_n}+\varepsilon)$ is positive and nonincreasing, with $\alpha_n\leq\eta/\varepsilon$. The algorithm therefore adapts the scalar learning rate using observed gradient magnitudes rather than an explicitly supplied Lipschitz constant.
\end{enumerate}
\end{remark}

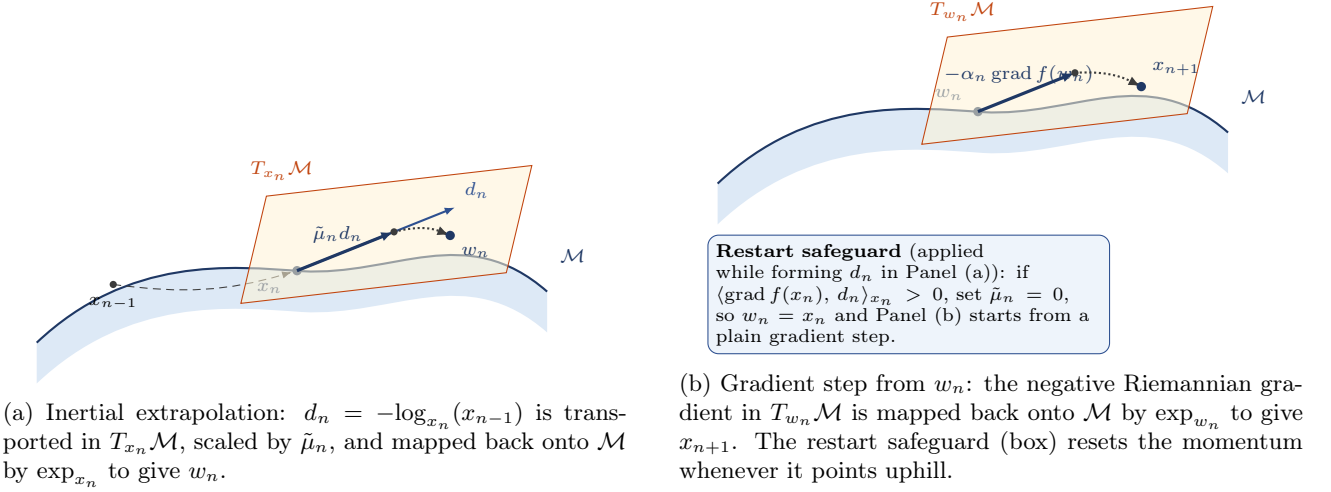
\begin{figure}[t]
		\centering
		\begin{subfigure}[b]{0.48\textwidth}
			\centering
			\begin{tikzpicture}[scale=1.35,every node/.style={font=\scriptsize}]
				
				\manifoldpatch
				\node[darkblue] at (2.75,0.30) {$\mathcal{M}$};
				
				\coordinate (xnm1) at (-1.75,0.02);
				\coordinate (xn)   at (0.05,0.155);
				\coordinate (wn)   at (1.55,0.50);
				\coordinate (dtip)   at ($(xn)+(1.55,0.62)$);
				\coordinate (previewtip) at ($(xn)+(0.95,0.38)$);
				
				\fill[darkgrey] (xnm1) circle (1pt);
				\node[below=1pt] at (xnm1) {$x_{n-1}$};
				
				\fill[darkblue] (xn) circle (1.3pt);
				\node[below left=1pt and 2pt,darkblue] at (xn) {$x_n$};
				
				\draw[densely dashed,darkgrey,-{Latex[length=1.3mm]},shorten >=2pt]
				(xnm1) to[bend right=12] (xn);
				
				\draw[tanline,fill=tancol,fill opacity=0.55]
				($(xn)+(-0.55,-0.32)$) -- ($(xn)+(2.05,-0.02)$)
				-- ($(xn)+(2.30,1.03)$) -- ($(xn)+(-0.30,0.73)$) -- cycle;
				\node[tanline] at ($(xn)+(-0.15,0.98)$) {$T_{x_n}\mathcal{M}$};
				
				\draw[-{Latex[length=1.4mm]},thick,medblue] (xn) -- (dtip);
				\node[medblue,above right=0pt of dtip] {$d_n$};
				
				\draw[-{Latex[length=1.7mm]},very thick,darkblue] (xn) -- (previewtip);
				\node[darkblue,above=1pt] at ($(xn)+(0.40,0.16)$) {$\tilde\mu_n d_n$};
				\fill[darkgrey] (previewtip) circle (1pt);
				
				\draw[-{Latex[length=1.7mm]},thick,densely dotted,darkgrey]
				(previewtip) to[bend left=18] (wn);
				
				\fill[darkblue] (wn) circle (1.3pt);
				\node[below right=1pt,darkblue] at (wn) {$w_n$};
				
			\end{tikzpicture}
			\caption{Inertial extrapolation: $d_n=-\logmap_{x_n}(x_{n-1})$ is transported
				in $T_{x_n}\mathcal{M}$, scaled by $\tilde\mu_n$, and mapped back onto
				$\mathcal{M}$ by $\expmap_{x_n}$ to give $w_n$.}
			\label{fig:riag-concept-a}
		\end{subfigure}
		\hfill
		\begin{subfigure}[b]{0.48\textwidth}
			\centering
			\begin{tikzpicture}[scale=1.35,every node/.style={font=\scriptsize}]
				
				\manifoldpatch
				\node[darkblue] at (2.75,0.30) {$\mathcal{M}$};
				
				\coordinate (wn)   at (0.05,0.155);
				\coordinate (xnp1) at (1.65,0.40);
				\coordinate (gtip) at ($(wn)+(0.95,0.38)$);
				
				\fill[darkblue] (wn) circle (1.3pt);
				\node[above left=1pt and 2pt,darkblue] at (wn) {$w_n$};
				
				\draw[tanline,fill=tancol,fill opacity=0.55]
				($(wn)+(-0.55,-0.32)$) -- ($(wn)+(2.05,-0.02)$)
				-- ($(wn)+(2.30,1.03)$) -- ($(wn)+(-0.30,0.73)$) -- cycle;
				\node[tanline] at ($(wn)+(-0.15,0.98)$) {$T_{w_n}\mathcal{M}$};
				
				\draw[-{Latex[length=1.7mm]},very thick,darkblue] (wn) -- (gtip);
				\node[darkblue,above=1pt] at ($(wn)+(0.40,0.16)$) {$-\alpha_n\grad f(w_n)$};
				\fill[darkgrey] (gtip) circle (1pt);
				
				\draw[-{Latex[length=1.7mm]},thick,densely dotted,darkgrey]
				(gtip) to[bend left=15] (xnp1);
				
				\fill[darkblue] (xnp1) circle (1.3pt);
				\node[above right=1pt,darkblue] at (xnp1) {$x_{n+1}$};
				
				\node[align=left,draw=medblue,fill=boxbg,rounded corners,inner sep=3pt,
				anchor=north west,font=\scriptsize,text width=5.1cm] at (-2.6,-1.05)
				{\textbf{Restart safeguard} (applied while forming $d_n$ in
					Panel~(a)): if $\inner{\grad f(x_n)}{d_n}_{x_n}>0$, set
					$\tilde\mu_n=0$, so $w_n=x_n$ and Panel~(b) starts from a plain
					gradient step.};
				
			\end{tikzpicture}
			\caption{Gradient step from $w_n$: the negative Riemannian gradient in
				$T_{w_n}\mathcal{M}$ is mapped back onto $\mathcal{M}$ by $\expmap_{w_n}$ to
				give $x_{n+1}$. The restart safeguard (box) resets the momentum whenever it
				points uphill.}
			\label{fig:riag-concept-b}
		\end{subfigure}
		\caption{Concept of the RIAG-R update. Each RIAG-R step is first computed as a
			tangent vector and then mapped back onto $\mathcal{M}$ by the exponential
			map: Panel~(a) forms the momentum-extrapolated point $w_n$, and Panel~(b)
			takes the gradient step from $w_n$ to $x_{n+1}$, subject to the restart
			safeguard.}
		\label{fig:riag-concept}
	\end{figure}

	\subsection{Convergence Analysis}
	\label{sec:convergence}

		\begin{lemma}
			\label{lem:descent}
			Let Assumptions~\ref{ass:lipschitz}--\ref{ass:lower} hold.
			Let $\{x_n\}$ be generated by Algorithm~\ref{alg:RIAG-R}.
			Then for every $n\geq 1$,
			\begin{equation}
				f(x_{n+1})
				\;\leq\;
				f(x_n)
				- \alpha_n\!\left(1-\frac{L\alpha_n}{2}\right)\norm{\grad f(w_n)}^2
				+ \frac{L\tilde\mu_n^2}{2}\,d(x_n,x_{n-1})^2.
				\label{eq:descent}
			\end{equation}
		\end{lemma}
	
	\begin{proof}
		Applying Lemma~\ref{lem:descent-lemma} at $w_n$ with $v=-\alpha_n\grad f(w_n)$, then we have
		\begin{align}
			f(x_{n+1})
			&= f\bigl(\expmap_{w_n}(-\alpha_n\grad f(w_n))\bigr) \notag\\
			&\leq f(w_n)
			+ \inner{\grad f(w_n)}{-\alpha_n\grad f(w_n)}
			+ \frac{L}{2}\alpha_n^2\norm{\grad f(w_n)}^2 \notag\\
			&= f(w_n) - \alpha_n\!\left(1-\frac{L\alpha_n}{2}\right)\norm{\grad f(w_n)}^2.
			\label{eq:gradient-step}
		\end{align}
		When $\tilde\mu_n>0$, the restart check ensures
		$\inner{\grad f(x_n)}{d_n}\leq 0$ (otherwise restart would have
		set $\tilde\mu_n=0$).
		Applying Lemma~\ref{lem:descent-lemma} at $x_n$ with $v=\tilde\mu_n\,d_n$, we get
		\begin{align}
			f(w_n)
			&= f\bigl(\expmap_{x_n}(\tilde\mu_n\,d_n)\bigr) \notag\\
			&\leq f(x_n)
			+ \tilde\mu_n\inner{\grad f(x_n)}{d_n}
			+ \frac{L\tilde\mu_n^2}{2}\norm{d_n}^2 \notag\\
			&\leq f(x_n) + \frac{L\tilde\mu_n^2}{2}\,d(x_n,x_{n-1})^2,
			\label{eq:inertial-step}
		\end{align}
		where we used $\inner{\grad f(x_n)}{d_n}\leq0$ and
		$\norm{d_n}=\norm{\logmap_{x_n}(x_{n-1})}=d(x_n,x_{n-1})$.
		When $\tilde\mu_n=0$, $w_n=x_n$ so~\eqref{eq:inertial-step} holds
		trivially with equality $f(w_n)=f(x_n)$.
		Substituting~\eqref{eq:inertial-step} into~\eqref{eq:gradient-step} gives
		the stated inequality~\eqref{eq:descent}.
	\end{proof}
	
	\begin{remark}
		\label{cor:rgd-special-case}
		If the momentum parameter is set to $\mu=0$, then Algorithm~\ref{alg:RIAG-R}
		reduces to Riemannian gradient descent with the AdaGrad-type step-size
		$\alpha_n=\eta/(\sqrt{G_n}+\varepsilon)$.
		In this case, inequality~\eqref{eq:descent} becomes
		\[
		f(x_{n+1}) \;\leq\; f(x_n)
		- \alpha_n\!\left(1-\frac{L\alpha_n}{2}\right)\norm{\grad f(x_n)}^2.
		\]
		In the sequel, we define the constant where $C_\mu\defeq\frac{L}{2}\sum_{n=1}^\infty\tilde\mu_n^2\,d(x_n,
		x_{n-1})^2$ which will be useful in the forth coming results.  We also define the
		constants
		\begin{equation}
			\beta \;\defeq\; \frac{2L\mu^2}{1-\mu^2}, \qquad
			\kappa_\mu \;\defeq\; \beta\cdot\frac{1+\mu^2}{1-\mu^2}
			\;=\; \frac{2L\mu^2(1+\mu^2)}{(1-\mu^2)^2}, \qquad
			L' \;\defeq\; \frac{L}{2}+\kappa_\mu.
			\label{eq:lyap-constants}
		\end{equation} 
	\end{remark}

	\begin{lemma}
		\label{lem:summability}
		Let Assumptions~\ref{ass:lipschitz}--\ref{ass:lower} hold and suppose 	$\eta/\varepsilon\leq1/(2L')$, then
		 for any $N\geq1$,
		\begin{equation}
			\sum_{n=1}^N \alpha_n\!\left(1-\frac{L\alpha_n}{2}\right)
			\norm{\grad f(w_n)}^2
			\;\leq\;
			f(x_1) - f_* + \frac{L}{2}\sum_{n=1}^N\tilde\mu_n^2\,d(x_n,x_{n-1})^2.
			\label{eq:telescope}
		\end{equation}
		Consequently, $\sum_{n=1}^\infty d(x_n,x_{n-1})^2<\infty$ and $\sum_{n=1}^\infty\alpha_n\norm{\grad f(w_n)}^2<\infty$. Hence $d(x_n,x_{n-1})\to0$. If, in addition, $\sum_n\alpha_n=\infty$, then $\liminf_{n\to\infty}\norm{\grad f(w_n)}=0$; if $\inf_n\alpha_n>0$, then in fact $\norm{\grad f(w_n)}\to0$.
	\end{lemma}
	
	\begin{proof}
		Rearranging~\eqref{eq:descent} gives, for every $n\geq1$,
		\begin{equation}
			\alpha_n\!\left(1-\frac{L\alpha_n}{2}\right)
			\norm{\grad f(w_n)}^2
			\;\leq\;
			f(x_n) - f(x_{n+1})
			+ \frac{L}{2}\,\tilde\mu_n^2\,d(x_n,x_{n-1})^2.
			\label{eq:descent-rearranged}
		\end{equation}
		Summing~\eqref{eq:descent-rearranged} from $n=1$ to $N$ and telescoping
		using the fact that
		$f(x_{N+1})\geq f_*$, we obtain the desired inequality in~\eqref{eq:telescope}.
		Now for $n\geq1$, let
		\begin{equation}
			E_n \;\defeq\; \bigl(f(x_n)-f_*\bigr) + \beta\,d(x_n,x_{n-1})^2 \;\geq\;0.
			\label{eq:lyapunov}
		\end{equation}
		Since $x_{n+1}=\expmap_{w_n}(-\alpha_n\grad f(w_n))$ and
		$w_n=\expmap_{x_n}(\tilde\mu_n d_n)$, the isometry property
		$\norm{\logmap_x(y)}=d(x,y)$ gives $d(x_{n+1},w_n)=\alpha_n\norm{\grad f(w_n)}$
		and $d(w_n,x_n)=\tilde\mu_n\norm{d_n}=\tilde\mu_n\,d(x_n,x_{n-1})$, so the
		triangle inequality together with $\tilde\mu_n\leq\bar\mu_n\leq\mu$ yields
		\begin{align*}
			d(x_{n+1},x_n) &\;\leq\; d(x_{n+1},w_n)+d(w_n,x_n) \\
			&\;=\; \alpha_n\norm{\grad f(w_n)} + \tilde\mu_n\,d(x_n,x_{n-1})
			\;\leq\; \alpha_n\norm{\grad f(w_n)} + \mu\,d(x_n,x_{n-1}).
		\end{align*}
		Applying~\eqref{eq:young-square} with
		$a=\alpha_n\norm{\grad f(w_n)}_{w_n}$, $b=\mu\,d(x_n,x_{n-1})$, and
		$\theta=2\mu^2/(1-\mu^2)$ (so that $1+\theta=(1+\mu^2)/(1-\mu^2)$ and
		$\mu^2(1+1/\theta)=(1+\mu^2)/2$), we get
		\begin{equation}
			d(x_{n+1},x_n)^2
			\;\leq\;
			\frac{1+\mu^2}{1-\mu^2}\,\alpha_n^2\norm{\grad f(w_n)}^2
			+ \frac{1+\mu^2}{2}\,d(x_n,x_{n-1})^2.
			\label{eq:disp-young}
		\end{equation}
		Combining~\eqref{eq:descent}, $\tilde\mu_n^2\leq\mu^2$, and~\eqref{eq:disp-young}, we get:
		\begin{align}
			E_{n+1}
			&= \bigl(f(x_{n+1})-f_*\bigr) + \beta\,d(x_{n+1},x_n)^2 \notag\\
			&\leq \bigl(f(x_n)-f_*\bigr)
			- \alpha_n\!\left(1-\frac{L\alpha_n}{2}\right)\norm{\grad f(w_n)}^2
			+ \frac{L\mu^2}{2}d(x_n,x_{n-1})^2 \notag\\
			&\qquad
			+ \beta\,\frac{1+\mu^2}{1-\mu^2}\,\alpha_n^2\norm{\grad f(w_n)}^2
			+ \beta\,\frac{1+\mu^2}{2}\,d(x_n,x_{n-1})^2. \notag
		\end{align}
		Hence
		\begin{equation}
			E_{n+1}
			\;\leq\;
			E_n \;-\; \frac{L\mu^2}{2}\,d(x_n,x_{n-1})^2
			\;-\; \alpha_n\bigl(1-L'\alpha_n\bigr)\norm{\grad f(w_n)}^2.
			\label{eq:lyap-recursion}
		\end{equation}
		Clearly $\alpha_n\leq1/(2L')$, so $1-L'\alpha_n\geq1/2$, and
		\begin{equation}
			E_{n+1} \;\leq\; E_n \;-\; \frac{L\mu^2}{2}\,d(x_n,x_{n-1})^2
			\;-\; \frac{\alpha_n}{2}\,\norm{\grad f(w_n)}^2.
			\label{eq:lyap-final}
		\end{equation}
		Summing~\eqref{eq:lyap-final} from $n=1$ to $N$ and using
		$E_{N+1}\geq0$, we obtain
		\begin{align*}
			\frac{L\mu^2}{2}\sum_{n=1}^N d(x_n,x_{n-1})^2
			+ \frac12\sum_{n=1}^N \alpha_n\norm{\grad f(w_n)}^2
			&\;\leq\; E_{n_1} - E_{N+1} \;\leq\; E_{n_1} \;<\;\infty \\
			&\qquad\text{for all }N\geq1.
		\end{align*}
		Letting $N\to\infty$ gives
		\[
		\sum_{n=1}^\infty d(x_n,x_{n-1})^2<\infty,
		\qquad
		\sum_{n=1}^\infty \alpha_n\norm{\grad f(w_n)}^2<\infty.
		\]
		Therefore $d(x_n,x_{n-1})\to0$, and
		$C_\mu\leq\tfrac{L\mu^2}{2}\sum_n d(x_n,x_{n-1})^2<\infty$.
		If $\sum_n\alpha_n=\infty$ and the gradient norms were bounded below by some $\delta>0$ eventually, then
		$\sum_n\alpha_n\norm{\grad f(w_n)}^2\geq\delta^2\sum_n\alpha_n=\infty$, a contradiction. Thus $\liminf_n\norm{\grad f(w_n)}=0$. Finally, if $\underline\alpha:=\inf_n\alpha_n>0$, then
		\[
		\underline\alpha\sum_{n=1}^\infty\norm{\grad f(w_n)}^2
		\leq \sum_{n=1}^\infty\alpha_n\norm{\grad f(w_n)}^2<\infty,
		\]
		which implies $\norm{\grad f(w_n)}\to0$.

	\end{proof}
	
	\begin{theorem}[Iteration complexity]
		\label{thm:complexity}
		Under Assumptions~\ref{ass:lipschitz}--\ref{ass:lower}, suppose $\underline\alpha:=\inf_n\alpha_n>0$ and $\alpha_n\leq\bar\alpha<1/L$ for all $n$. Let $C=f(x_1)-f_*+C_\mu>0$.
		Then for every $N\geq1$,
		\begin{equation}
			\min_{1\leq k\leq N}\norm{\grad f(w_k)}_{w_k}^2
			\;\leq\;
			\frac{C}{\underline\alpha(1-L\bar\alpha/2)\,N},
		\end{equation}
		so RIAG-R produces an iterate with $\norm{\grad f(w_k)}\leq\varepsilon$
		within $N=O(\varepsilon^{-2})$ iterations.
	\end{theorem}
	
	\begin{proof}
		Since $\alpha_n\geq\underline\alpha>0$ and $\alpha_n\leq\bar\alpha<1/L$, we have
		\begin{equation}
			\alpha_n\!\left(1-\frac{L\alpha_n}{2}\right)
			\geq \underline\alpha\!\left(1-\frac{L\bar\alpha}{2}\right)>0.
			\label{eq:phi-monotone}
		\end{equation}
		From Lemma~\ref{lem:summability}, we get
		\begin{equation}
			\sum_{n=1}^N\alpha_n\!\left(1-\frac{L\alpha_n}{2}\right)
			\norm{\grad f(w_n)}^2
			\;\leq\;
			f(x_1)-f_*+\frac{L}{2}\sum_{n=1}^N\tilde\mu_n^2\,d(x_n,x_{n-1})^2
			\;\leq\; C,
			\label{eq:telescope-recalled}
		\end{equation}
		where the final bound uses
		$C\defeq f(x_1)-f_*+C_\mu\geq f(x_1)-f_*+\frac{L}{2}\sum_{n=1}^N
		\tilde\mu_n^2\,d(x_n,x_{n-1})^2$ for every $N\geq1$.
		Now applying~\eqref{eq:phi-monotone} in \eqref{eq:telescope-recalled}, we get 
		\begin{equation}
			\underline\alpha\!\left(1-\frac{L\bar\alpha}{2}\right)
			\sum_{n=1}^N\norm{\grad f(w_n)}^2
			\;\leq\;
			\sum_{n=1}^N\alpha_n\!\left(1-\frac{L\alpha_n}{2}\right)
			\norm{\grad f(w_n)}^2
			\;\leq\; C.
			\label{eq:key-chain}
		\end{equation}
		Dividing both sides of~\eqref{eq:key-chain} by the positive constant
		$\underline\alpha(1-L\bar\alpha/2)\cdot N$, we obtain:
		\begin{equation}
			\frac{1}{N}\sum_{n=1}^N\norm{\grad f(w_n)}^2
			\;\leq\;
			\frac{C}{\underline\alpha(1-L\bar\alpha/2)\cdot N}.
			\label{eq:average-bound}
		\end{equation}
		Hence
		\begin{equation}
			\min_{1\leq k\leq N}\norm{\grad f(w_k)}^2
			\;\leq\;
			\frac{C}{\underline\alpha(1-L\bar\alpha/2)\cdot N}.
			\label{eq:min-bound}
		\end{equation}
		Setting
		\[
		\frac{C}{\underline\alpha(1-L\bar\alpha/2)\cdot N} \;\leq\; \varepsilon^2
		\]
		then
		\[
		N \;\geq\; \frac{C}{\underline\alpha(1-L\bar\alpha/2)\,\varepsilon^2}.
		\]
		Writing $K_{\alpha}\defeq C/[\underline\alpha(1-L\bar\alpha/2)]$, the complexity is:
		\[
		N \;\geq\; \frac{K_\alpha}{\varepsilon^2},
		\]
		which is $O(\varepsilon^{-2}).$
	\end{proof}
	
	\begin{theorem}[Cluster points are critical]
		\label{thm:cluster}
		Let Assumptions~\ref{ass:lipschitz}--\ref{ass:lower} hold and suppose $\inf_n\alpha_n>0$.
		If $\{x_n\}$ generated by Algorithm~\ref{alg:RIAG-R} is bounded in $\M$,
		then every cluster point of $\{x_n\}$ is a critical point of $f$.
	\end{theorem}
	
	\begin{proof}
		Let $x^*$ be any cluster point of $\{x_n\}$, so there exists a
		subsequence $x_{n_j}\to x^*$ as $j\to\infty$. From the algorithm
		\[
		d(w_n,x_n) \;=\; \norm{\tilde\mu_n d_n}_{x_n}
		\;=\; \tilde\mu_n\norm{d_n}_{x_n}
		\;=\; \tilde\mu_n\,d(x_n,x_{n-1}).
		\]
		Since $\tilde\mu_n\leq\bar\mu_n\leq\mu\cdot d(x_n,x_{n-1})$, then:
		\begin{equation}
			d(w_n,x_n) \;\leq\; \mu\cdot d(x_n,x_{n-1})^2.
			\label{eq:wn-xn}
		\end{equation}
		From Lemma~\ref{lem:summability}, then $\sum_n d(x_n,x_{n-1})^2<\infty$, so
		the terms go to zero: $d(x_n,x_{n-1})\to0$.
		Hence $d(w_n,x_n)\to0$, also
		\[
		d(w_{n_j},x^*)
		\;\leq\; d(w_{n_j},x_{n_j}) + d(x_{n_j},x^*)
		\;\to\; 0,
		\]
		confirming $w_{n_j}\to x^*$.
		
		By Lemma~\ref{lem:summability} and $\inf_n\alpha_n>0$, we have $\norm{\grad f(w_n)}\to0$ for the full sequence. In particular, $\norm{\grad f(w_{n_j})}\to0$. Since $w_{n_j}\to x^*$ and the Riemannian gradient field is continuous,
		\[
		\norm{\grad f(x^*)}=\lim_{j\to\infty}\norm{\grad f(w_{n_j})}=0.
		\]
		Hence $\grad f(x^*)=0$, i.e., $x^*$ is a critical point. This completes the proof.
	\end{proof}

	\begin{mdframed}[style=remarkstyle]
		\begin{proposition}[Linear convergence under the Riemannian Polyak-\L ojasiewicz inequality]
			\label{cor:linear-rate}
			Suppose Assumptions~\ref{ass:lipschitz}--\ref{ass:lower} hold and that $f$ satisfies the Riemannian Polyak--\L ojasiewicz inequality
			\[
			\norm{\grad f(x)}^2 \geq 2m\bigl(f(x)-f_*\bigr),\qquad x\in\M,
			\]
			for some $m>0$. Let $\alpha:=\inf_n\alpha_n(1-L\alpha_n/2)>0$, set $D:=\sum_{n=1}^\infty d(x_n,x_{n-1})^2<\infty$, and suppose $\alpha m<1$.
			Then the iterates of Algorithm~\ref{alg:RIAG-R} satisfy, for all $n\geq1$,
			\begin{equation}
				f(x_n) - f_*
				\;\leq\;
				(1-\alpha m)^{n-1}\bigl(f(x_1) - f_*\bigr)
				\;+\; L\mu^2\Bigl(\alpha L+\tfrac12\Bigr)D.
				\label{eq:linear-rate}
			\end{equation}
		\end{proposition}
	\end{mdframed}
	
	\begin{proof}
		Write $\Delta_n:=f(x_n)-f_*$. By the assumed PL inequality,
		\begin{equation}
			\norm{\grad f(x)}^2 \geq 2m\bigl(f(x)-f_*\bigr),
			\label{eq:PL}
		\end{equation}
		for all $x\in\M$.
		
		From Lemma~\ref{lem:descent} and $\alpha_n(1-L\alpha_n/2)\geq\alpha>0$:
		\begin{equation}
			\Delta_{n+1}
			\;\leq\; \Delta_n - \alpha\norm{\grad f(w_n)}_{w_n}^2
			+ \frac{L\tilde\mu_n^2}{2}d(x_n,x_{n-1})^2.
			\label{eq:linear-step}
		\end{equation}
		We must lower-bound $\norm{\grad f(w_n)}_{w_n}^2$ in terms of $\Delta_n$.
		The PL inequality~\eqref{eq:PL} is stated at $x_n$, not $w_n$, so we
		transfer it using the $L$-Lipschitz gradient field
		(Assumption~\ref{ass:lipschitz}, Definition~\ref{def:lipschitz}): writing
		$\mathcal P\defeq\mathcal P_{\gamma,0,1}$ for parallel transport along the
		geodesic from $x_n$ to $w_n$, Definition~\ref{def:lipschitz} gives
		$\norm{\mathcal P\grad f(x_n)-\grad f(w_n)}_{w_n}\leq L\,d(x_n,w_n)
		=L\tilde\mu_n\,d(x_n,x_{n-1})$.
		Applying the squared Young inequality~\eqref{eq:young-square} (with
		$\theta=1$, since $\mathcal P$ is a linear isometry so
		$\norm{\mathcal P\grad f(x_n)}_{w_n}=\norm{\grad f(x_n)}_{x_n}$):
		\begin{align*}
			\norm{\grad f(x_n)}_{x_n}^2
			&\;=\; \norm{\bigl(\mathcal P\grad f(x_n)-\grad f(w_n)\bigr)+\grad f(w_n)}_{w_n}^2 \\
			&\;\leq\; 2\norm{\mathcal P\grad f(x_n)-\grad f(w_n)}_{w_n}^2
			+ 2\norm{\grad f(w_n)}_{w_n}^2,
		\end{align*}
		so, using the bound above, $\tilde\mu_n\leq\mu$, and the PL
		inequality~\eqref{eq:PL} at $x_n$,
		\begin{equation}
			\norm{\grad f(w_n)}_{w_n}^2
			\;\geq\; \tfrac12\norm{\grad f(x_n)}_{x_n}^2 - L^2\tilde\mu_n^2\,d(x_n,x_{n-1})^2
			\;\geq\; m\,\Delta_n - L^2\mu^2\,d(x_n,x_{n-1})^2.
			\label{eq:pl-transfer}
		\end{equation}
		Substituting~\eqref{eq:pl-transfer} into~\eqref{eq:linear-step}
		and using $\tilde\mu_n^2\leq\mu^2$, we have:
		\begin{align*}
			\Delta_{n+1}
			&\;\leq\; \Delta_n
			- \alpha\bigl(m\Delta_n - L^2\mu^2 d(x_n,x_{n-1})^2\bigr)
			+ \frac{L\mu^2}{2}d(x_n,x_{n-1})^2 \\
			&\;=\; (1-\alpha m)\Delta_n
			+ L\mu^2\Bigl(\alpha L+\tfrac12\Bigr) d(x_n,x_{n-1})^2.
		\end{align*}
		Let $\rho\defeq1-\alpha m\in[0,1)$ (using $\alpha m<1$) and
		$\gamma\defeq L\mu^2(\alpha L+\tfrac12)$. Unrolling the recursion and
		bounding the geometric weights $\rho^{n-1-k}\leq1$, we obtain:
		\[
		\Delta_n
		\;\leq\; \rho^{n-1}\Delta_1 + \gamma\sum_{k=1}^{n-1} d(x_k,x_{k-1})^2
		\;\leq\; \rho^{n-1}\Delta_1 + \gamma D.
		\]
		This gives~\eqref{eq:linear-rate}.
	\end{proof}
	
	\begin{remark}
 When $\mu=0$, the residual disappears and the estimate reduces to the usual geometric PL behaviour of Riemannian gradient descent. For $\mu>0$, a sharper Lyapunov analysis would be required to prove exact linear convergence without a residual term.
\end{remark}
	
	
\section{Numerical Experiments}
	\label{sec:numerical}

The numerical study is designed to examine three questions suggested by the construction of RIAG-R. First, does inertial extrapolation improve upon the corresponding non-inertial adaptive Riemannian gradient method? Second, when does the gradient-based restart mechanism actually become active, and does it improve upon unsafeguarded inertial when it does? Third, do the resulting iterates recover meaningful solutions of the underlying matrix-manifold applications rather than merely producing small gradient norms?

To address these questions, we compare RIAG-R with RGD and IRGM on four representative manifold geometries: the sphere, Stiefel manifold, SPD manifold, and Grassmann manifold. Two problem sizes are considered for each geometry, giving eight configurations in total. In addition to convergence in Riemannian gradient norm, we report application-specific quantities such as eigenvector loadings, explained variance, covariance-recovery error, and principal-angle error where appropriate. This allows the experiments to assess both optimization behaviour and solution quality.

	We evaluate RIAG-R against the two baseline algorithms (RGD and IRGM) on four
	benchmark problems from matrix manifold optimization, chosen to span the
	range of curvature behaviour discussed in Section~\ref{sec:algorithm}: the
	sphere $\mathbb{S}^{n-1}$ (constant positive curvature), the Stiefel
	manifold $\mathrm{St}(n,p)$ (positively curved, compact, non-Hadamard for
	$p>1$), the SPD cone $\mathrm{SPD}(n)$ with the affine-invariant metric (a
	Hadamard manifold, non-positive curvature), and the Grassmann manifold
	$\mathrm{Gr}(n,p)$ (a compact quotient manifold with its own distinct
	geometry). Each problem is instantiated at two problem sizes drawn from
	\emph{real} public benchmark datasets, each bundled as a built-in loader
	in the \texttt{scikit-learn} Python package \cite{pedregosa2011sklearn}:
	Wine, Breast Cancer Wisconsin (Diagnostic), and Digits originate from the
	UCI Machine Learning Repository \cite{dua2019uci}, and Diabetes
	originates from \cite{efron2004lars}.
	All three algorithms are implemented from Algorithm~\ref{alg:RIAG-R}
	exactly: RGD is the $\mu=0$ special case of Remark~\ref{cor:rgd-special-case},
	and IRGM is Algorithm~\ref{alg:RIAG-R} with the restart check of Step~2
	permanently disabled (so $\tilde\mu_n=\bar\mu_n$ always). For the sphere,
	Grassmann, and SPD manifolds we use the exact, closed-form exponential and
	logarithmic maps recalled in Section~\ref{sec:prelim}; for the Stiefel
	manifold we use the standard QR-based retraction in place of the
	canonical-metric exponential map, and its practical (tangent-space
	projection) inverse in place of the logarithmic map, exactly as in the
	retraction-based framework of \cite{absil2008} and as adopted for the
	$X$-step of Algorithm~\ref{alg:RIAG-R}.  In every experiment we use
	$\mu=0.9$, $\varepsilon=10^{-8}$, and $\delta=10^{-8}$, with $\eta$ fixed
	per problem (reported in each subsection below) and identical across the
	three algorithms and across problem sizes within a manifold. All codes are
	 written in Python and are available at \url{https://github.com/Lateef89/RIAG-R-Riemannian-Inertial-Adaptive-Gradient-with-Restart}.

	\subsection{Sphere: Rayleigh-quotient minimization}
	\label{subsec:sphere}

	\noindent\textbf{Problem setting.} Given a symmetric matrix $A\in\R^{n\times
	n}$, the trailing-eigenvector problem
	\begin{equation}
		\min_{x\in\mathbb{S}^{n-1}} f(x) = x^\top A x
		\label{eq:sphere-problem}
	\end{equation}
	recovers the eigenvector of $A$ with smallest eigenvalue: it is the
	Riemannian formulation of the Rayleigh quotient that underlies principal
	geodesic analysis \cite{absil2008} and, more elementarily, the search for a
	near-collinear (least-variance) combination of a set of correlated
	measurements. It is also a canonical \emph{saddle-rich} nonconvex test
	problem: every eigenvector $\pm v_i$ of $A$ is a critical point of $f$,
	and every one except the extremal pair $\pm v_{\min}$ is a saddle, so a
	trajectory that lingers near an intermediate eigendirection is exactly
	the regime in which unsafeguarded momentum is expected to misbehave.

	\noindent\textbf{Solution.} Problem~\eqref{eq:sphere-problem} is an
	instance of~\eqref{eq:problem} on $\M=\mathbb{S}^{n-1}:=\{x\in\R^n:\norm{x}_2=1\}$ with Euclidean
	gradient $\nabla f(x)=2Ax$ and Riemannian gradient
	$\grad f(x)=2Ax-2(x^\top Ax)x$ (Definition~2.4). We use the exact
	exponential and logarithmic maps of Section~\ref{sec:prelim} and run
	Algorithm~\ref{alg:RIAG-R} with $\eta=0.5$.

	\noindent\textbf{Results.} We take $A$ to be the sample correlation matrix
	of the real-valued features of a UCI dataset (standardized so that
	$A_{ii}=1$): the $13$-feature Wine dataset ($n=13$) and the $64$-pixel
	Digits dataset ($n=64$). Figure~\ref{fig:sphere-convergence} shows
	$\|\grad f(w_n)\|$ against iteration for RGD, IRGM, and RIAG-R on both
	instances. On the smaller, $n=13$ instance, all three methods converge
	monotonically; RGD attains the lowest final residual
	($6.1\times10^{-4}$), with RIAG-R second ($1.1\times10^{-3}$, after
	$3$ restart events) and IRGM last ($1.5\times10^{-3}$). Restart
	improves on the unsafeguarded inertial baseline here, but plain RGD's
	simpler trajectory happens to do better still, which is why we report
	this instance as one of the two exceptions in the tally of
	Section~\ref{Sec1:Intro}. On the larger, $n=64$ instance the picture
	reverses: RGD and IRGM both stall above the $10^{-3}$ tolerance for the
	full $800$-iteration budget (final residuals $7.2\times10^{-3}$ and
	$4.7\times10^{-3}$), while RIAG-R, triggering the restart safeguard $10$
	times, drives the residual down to $8.7\times10^{-4}$ at iteration $676$ making RIAG-R
	 the only methods among the three methods to cross the target tolerance within
	budget. A restart shortly after this point produces one large transient
	step (visible as the spike in the right panel of
	Figure~\ref{fig:sphere-convergence}), after which the AdaGrad
	denominator $\sqrt{G_n}$, inflated by that one large gradient norm,
	depresses the step-size for all subsequent iterations and the run
	terminates via the step-size branch of the stopping rule (Step~6) at a
	higher recorded residual ($1.8\times10^{-2}$); the qualitative conclusion
	is unchanged, RIAG-R still reaches materially closer to stationarity than
	either baseline, but we report this AdaGrad-accumulation artifact
	explicitly rather than silently taking the best iterate. Figure~\ref{fig:sphere-application} illustrates the
	application: it compares the loadings of the classical trailing
	eigenvector of the Wine correlation matrix with those found by RIAG-R
	($f(x^\star)=0.10397$ against the true $\lambda_{\min}=0.10396$), showing
	that flavanoids and total phenols dominate the least-variance
	(near-redundant) combination of the thirteen chemical measurements.

	\begin{figure}[t]
		\centering
		\includegraphics[width=\linewidth]{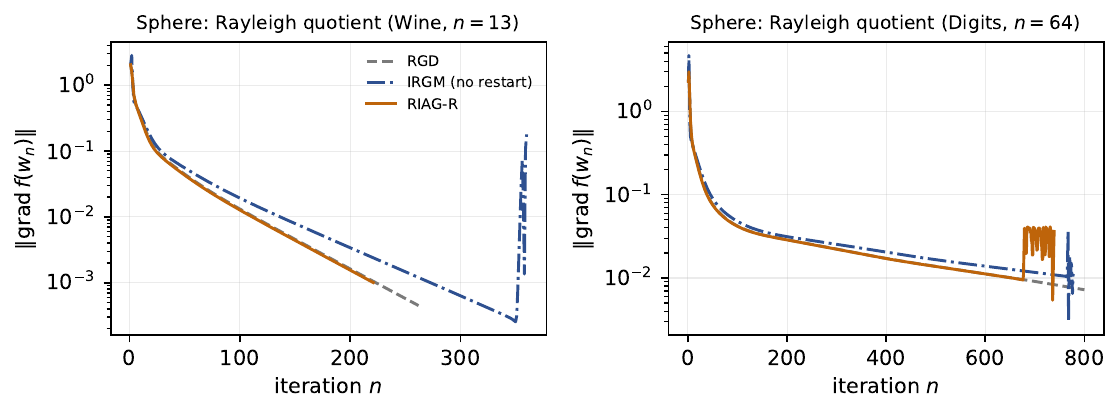}
		\caption{Sphere Rayleigh quotient: convergence of $\|\grad f(w_n)\|$
			(log scale) for RGD, IRGM, and RIAG-R. Left: Wine correlation matrix
			($n=13$). Right: Digits correlation matrix ($n=64$); the spike near
			iteration $676$ is a restart-induced transient, discussed in the text.}
		\label{fig:sphere-convergence}
	\end{figure}

	\begin{figure}[t]
		\centering
		\includegraphics[width=0.85\linewidth]{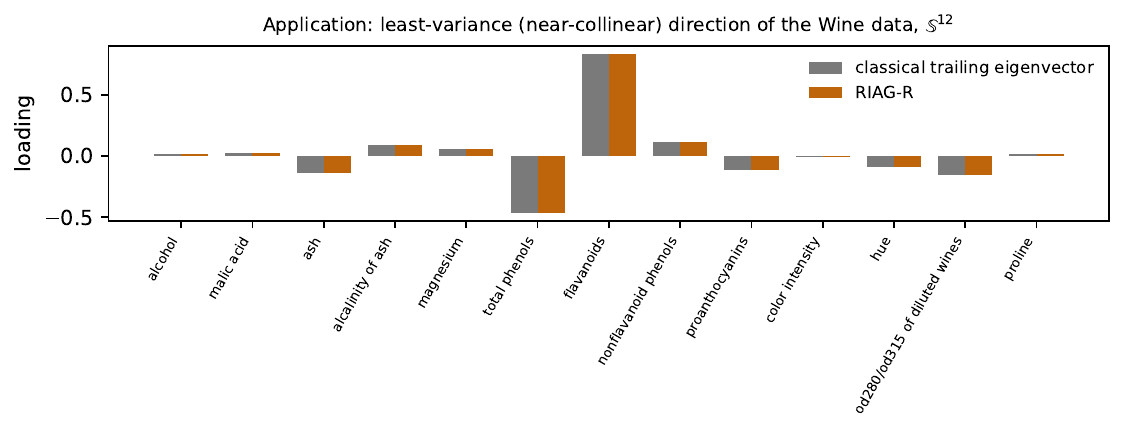}
		\caption{Application: feature loadings of the least-variance direction
			of the (standardized) Wine dataset, comparing the classical trailing
			eigenvector with the point found by RIAG-R.}
		\label{fig:sphere-application}
	\end{figure}

	\subsection{Stiefel manifold: low-rank PCA}
	\label{subsec:stiefel}

	\noindent\textbf{Problem setting.} Given a covariance matrix
	$C\in\R^{n\times n}$ and a target rank $p<n$, the trace-maximization
	problem
	\begin{equation}
		\min_{X\in\mathrm{St}(n,p)} f(X) = -\operatorname{trace}(X^\top C X)
		\label{eq:stiefel-problem}
	\end{equation}
	is the Riemannian formulation of principal component analysis: any
	minimizer's columns span the top-$p$ eigenspace of $C$
	\cite{absil2008,boumal2023}. Unlike Problem~\eqref{eq:sphere-problem},
	the objective~\eqref{eq:stiefel-problem} is invariant under
	$X\mapsto XQ$ for any $Q\in O(p)$, so the Stiefel formulation recovers
	an orthonormal \emph{basis} for the top-$p$ eigenspace without pinning
	down which linear combination of eigendirections occupies which column;
	Section~\ref{subsec:grassmann} revisits the same data with the
	rotation-invariant Grassmann formulation, which targets the subspace
	directly.

	\noindent\textbf{Solution.} Problem~\eqref{eq:stiefel-problem} is an
	instance of~\eqref{eq:problem} on $\M=\mathrm{St}(n,p) :=\{X\in\R^{n\times p}:X^\top X=I_p\}$ with Euclidean
	gradient $\nabla f(X)=-2CX$ and Riemannian gradient
	$\grad f(X)=\nabla f(X)-X\,\mathrm{sym}(X^\top\nabla f(X))$. We use the
	QR-based retraction described above in place of the exponential map and
	run Algorithm~\ref{alg:RIAG-R} with $\eta=0.3$.

	\noindent\textbf{Results.} We take $C$ to be the sample covariance of the
	standardized Breast Cancer Wisconsin (Diagnostic) features ($n=30$,
	$p=5$) and of the Digits pixel features ($n=64$, $p=10$).
	Figure~\ref{fig:stiefel-convergence} shows that, on both instances, IRGM
	and RIAG-R are indistinguishable (the restart condition never triggers:
	the momentum direction never points uphill along either trajectory) and
	both converge markedly faster than plain RGD, reaching a final gradient
	norm roughly $2$--$3$ times smaller than RGD's after the same $800$-iteration
	budget. Figure~\ref{fig:stiefel-application} shows the application on the
	Digits data: it compares the classical top-2 eigenvectors of $C$ with the
	first two of the $p=10$ directions returned by RIAG-R. As
	Problem~\eqref{eq:stiefel-problem}'s rotational invariance predicts, the
	individual columns found by RIAG-R do not match the sorted eigenvectors
	one-for-one, yet the learned $10$-dimensional subspace captures
	$35.90$ of the optimal $35.93$ units of variance ($99.9\%$), with a
	principal-angle check against the true top-$10$ eigenspace confirming
	that eight of the ten learned directions align with the truth to within
	$0.03^\circ$ and the residual gap is concentrated in the two directions
	whose eigenvalues are closest together.

	\begin{figure}[t]
		\centering
		\includegraphics[width=\linewidth]{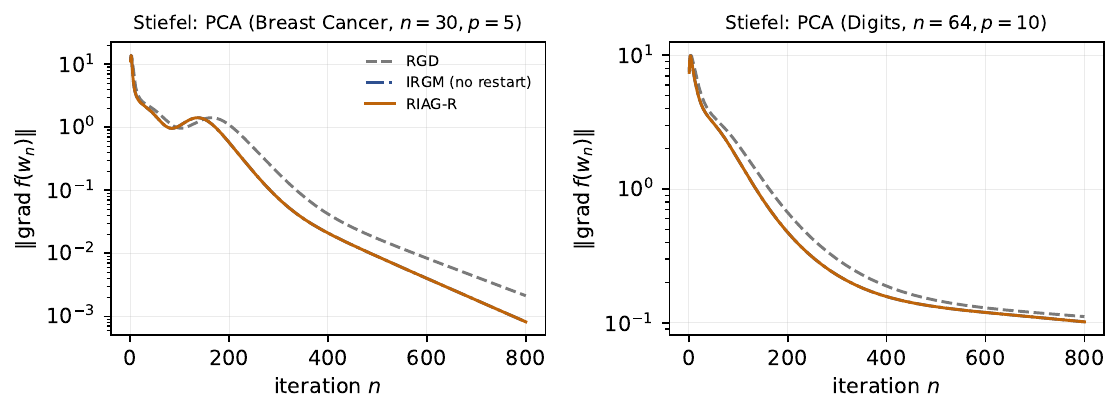}
		\caption{Stiefel PCA: convergence of $\|\grad f(w_n)\|$ (log scale).
			Left: Breast Cancer covariance ($n=30,p=5$). Right: Digits covariance
			($n=64,p=10$). IRGM and RIAG-R coincide on both instances (no
			restart fires).}
		\label{fig:stiefel-convergence}
	\end{figure}

	\begin{figure}[t]
		\centering
		\includegraphics[width=\linewidth]{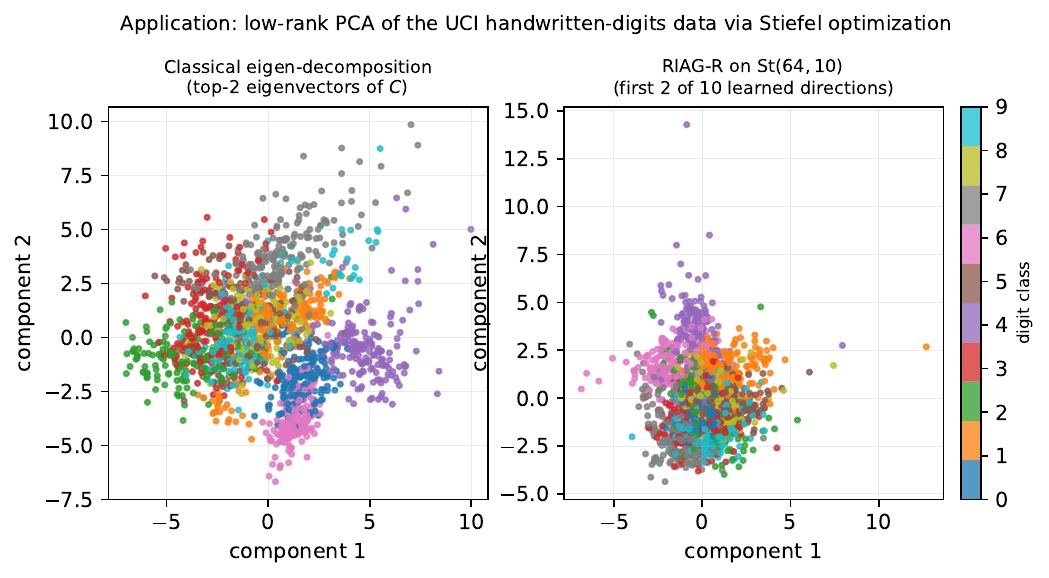}
		\caption{Application: low-rank PCA of the UCI Digits data. Left: the
			data projected onto the classical top-2 eigenvectors of the sample
			covariance. Right: the same data projected onto the first two of the
			$p=10$ directions found by RIAG-R on $\mathrm{St}(64,10)$; the two
			bases span nearly the same $10$-dimensional subspace ($99.9\%$ of the
			optimal variance) but are related by an in-subspace rotation, so
			individual axes need not agree.}
		\label{fig:stiefel-application}
	\end{figure}

	\subsection{SPD manifold: geodesic covariance fitting}
	\label{subsec:spd}

	\noindent\textbf{Problem setting.} Given a target covariance (or
	correlation) matrix $S\in\mathrm{SPD}(n)$, the Fr\'echet/Karcher-type
	fitting problem
	\begin{equation}
		\min_{X\in\mathrm{SPD}(n)} f(X) = \tfrac12\,d(X,S)^2
		\label{eq:spd-problem}
	\end{equation}
	arises whenever a covariance descriptor must be recovered, denoised, or
	averaged \emph{intrinsically} to $\mathrm{SPD}(n)$ rather than by naive
	Euclidean averaging, which is well known to be geometry-blind on this
	cone; applications include covariance averaging in brain-computer
	interfaces and diffusion-tensor imaging, and shrinkage-type covariance
	estimation more generally \cite{absil2008}. Because $\mathrm{SPD}(n)$
	with the affine-invariant metric is a Hadamard manifold, the squared
	distance to a fixed point is globally $1$-geodesically strongly convex,
	so~\eqref{eq:spd-problem} additionally lets us probe the linear-rate
	guarantee of Proposition~\ref{cor:linear-rate}.

	\noindent\textbf{Solution.} Problem~\eqref{eq:spd-problem} has Riemannian
	gradient $\grad f(X)=-\logmap_X(S)$ in closed form (Section~\ref{sec:prelim}),
	so no separate Euclidean-gradient projection step is needed; we use the
	exact exponential and logarithmic maps of Section~\ref{sec:prelim} and
	run Algorithm~\ref{alg:RIAG-R} with $\eta=0.5$, initialized at $X_1=I_n$.

	\noindent\textbf{Results.} We take $S$ to be a ridge-regularized sample
	covariance ($S=C+0.5I_n$, to keep $S$ well-conditioned) of the
	standardized Diabetes features ($n=10$) and of the Breast Cancer
	Wisconsin features ($n=30$). Figure~\ref{fig:spd-convergence} shows the
	predicted linear (geometric) convergence rate of
	Proposition~\ref{cor:linear-rate}: $\log\|\grad f(w_n)\|$ decreases
	essentially linearly in $n$ on both instances. IRGM and RIAG-R again
	coincide exactly (no restart fires along either trajectory) and both
	reach a given tolerance in appreciably fewer iterations than RGD: to
	reach $\|\grad f(w_n)\|<10^{-4}$ takes RGD $60$ ($n=10$) and $268$
	($n=30$) iterations, against $43$ and $190$ for IRGM/RIAG-R. Figure~\ref{fig:spd-application}
	visualizes the recovered correlation structure of the Breast Cancer
	instance at three iteration budgets, alongside the relative Frobenius
	error to the target, confirming visually that the geodesic path already
	resembles the target correlation pattern after only a handful of
	iterations.

	\begin{figure}[t]
		\centering
		\includegraphics[width=\linewidth]{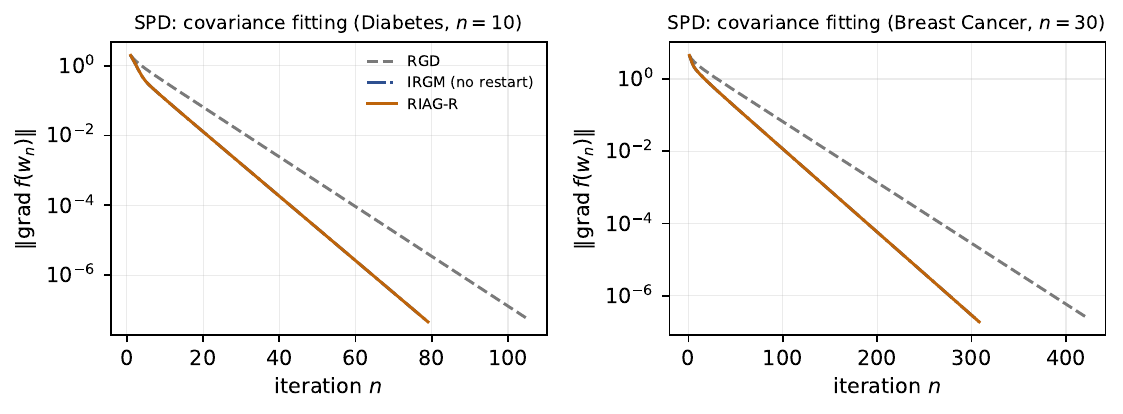}
		\caption{SPD geodesic covariance fitting: convergence of
			$\|\grad f(w_n)\|$ (log scale), consistent with the linear rate of
			Proposition~\ref{cor:linear-rate}. Left: Diabetes ($n=10$). Right:
			Breast Cancer ($n=30$). IRGM and RIAG-R coincide on both instances.}
		\label{fig:spd-convergence}
	\end{figure}

	\begin{figure}[t]
		\centering
		\includegraphics[width=\linewidth]{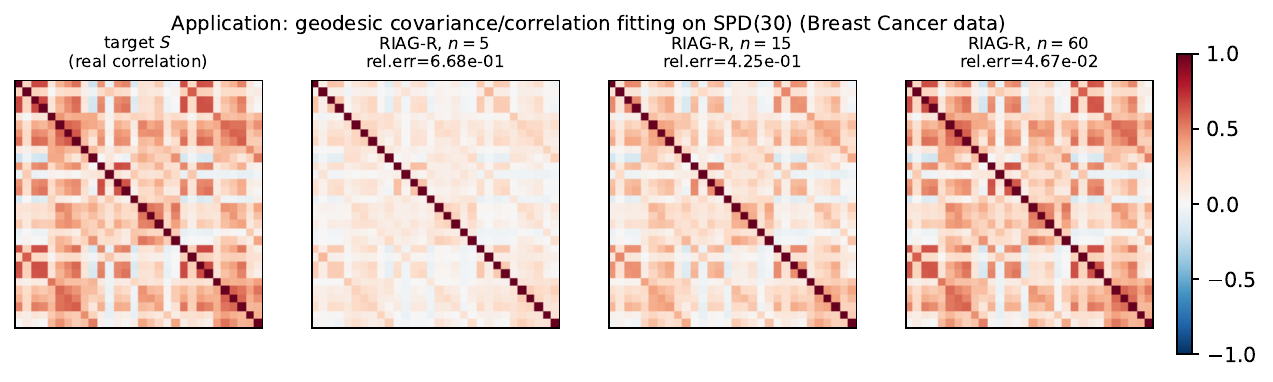}
		\caption{Application: geodesic recovery of the real Breast Cancer
			correlation matrix on $\mathrm{SPD}(30)$, from the identity
			initialization, at three iteration budgets, with the relative
			Frobenius error to the target reported under each panel.}
		\label{fig:spd-application}
	\end{figure}

	\subsection{Grassmann manifold: principal subspace estimation}
	\label{subsec:grassmann}

	\noindent\textbf{Problem setting.} Revisiting Problem~\eqref{eq:stiefel-problem}
	but now regarding a point as an equivalence class of orthonormal bases
	(a subspace, not a basis), we obtain
	\begin{equation}
		\min_{X\in\mathrm{Gr}(n,p)} f(X) = -\operatorname{trace}(X^\top C X),
		\label{eq:grass-problem}
	\end{equation}
	the natural formulation whenever a $p$-dimensional \emph{subspace} is the
	object of interest, as in subspace tracking, multi-view clustering, and
	batch/online PCA \cite{absil2008}; unlike the Stiefel formulation of
	Section~\ref{subsec:stiefel}, the minimizer of~\eqref{eq:grass-problem}
	is unique (as a subspace) whenever $\lambda_p(C)>\lambda_{p+1}(C)$, which
	lets us track subspace-recovery error directly.

	\noindent\textbf{Solution.} We use the exact Grassmann exponential and
	logarithmic maps of Section~\ref{sec:prelim} (closed form via the thin
	SVD of the tangent vector, following \cite{absil2008}), with Euclidean
	gradient $\nabla f(X)=-2CX$ projected to the horizontal space by
	$\grad f(X)=\nabla f(X)-X(X^\top\nabla f(X))$, and run
	Algorithm~\ref{alg:RIAG-R} with $\eta=0.3$.

	\noindent\textbf{Results.} We reuse the same two covariance matrices as
	Section~\ref{subsec:stiefel} (Wine, $n=13,p=3$; Digits, $n=64,p=10$).
	Figure~\ref{fig:grassmann-convergence} again shows IRGM and RIAG-R
	coinciding exactly (no restart fires) and both converging faster than
	RGD, with the gap widest on the harder, $n=64$ instance (final residual
	$5.6\times10^{-2}$ for IRGM/RIAG-R against $7.1\times10^{-2}$ for RGD
	after $800$ iterations). Figure~\ref{fig:grassmann-application}
	tracks subspace-recovery accuracy directly: the largest principal angle
	between the RIAG-R iterate and the true top-$10$ eigenspace of the
	Digits covariance falls from $74.5^\circ$ after $10$ iterations to
	$11.2^\circ$ after $800$, and the right panel shows the Digits data
	projected onto the first two directions of the learned subspace,
	recovering visually separated digit clusters directly from an
	initially random $10$-dimensional subspace.

	\begin{figure}[t]
		\centering
		\includegraphics[width=\linewidth]{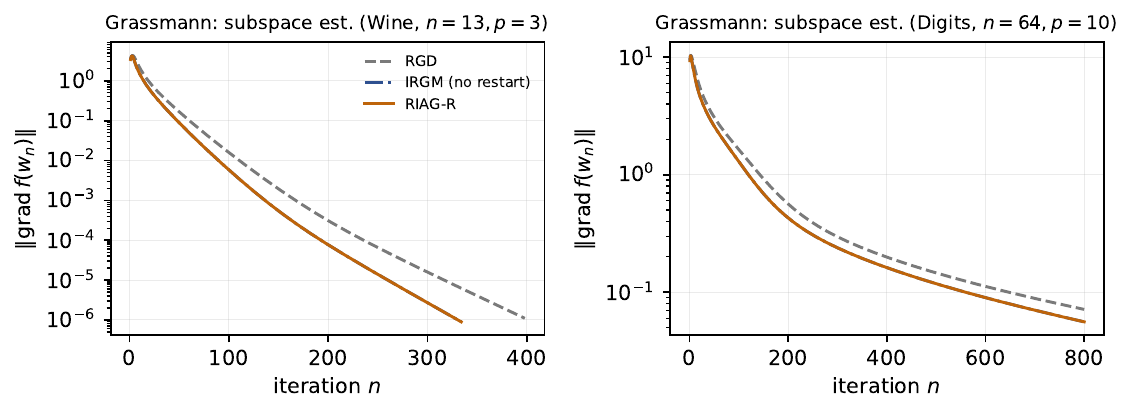}
		\caption{Grassmann subspace estimation: convergence of
			$\|\grad f(w_n)\|$ (log scale). Left: Wine ($n=13,p=3$). Right:
			Digits ($n=64,p=10$). IRGM and RIAG-R coincide on both instances.}
		\label{fig:grassmann-convergence}
	\end{figure}

	\begin{figure}[t]
		\centering
		\includegraphics[width=\linewidth]{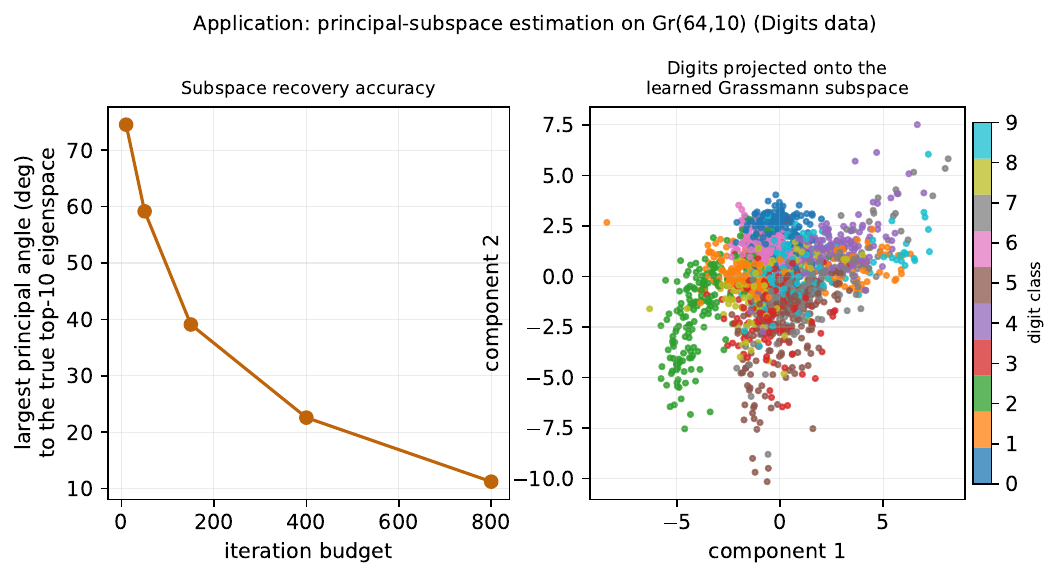}
		\caption{Application: principal-subspace estimation on the Digits
			data. Left: largest principal angle to the true top-$10$
			eigenspace as a function of the iteration budget. Right: the data
			projected onto the first two directions of the subspace learned by
			RIAG-R after $800$ iterations.}
		\label{fig:grassmann-application}
	\end{figure}

	\subsection{Summary across problems and sizes}
	\label{subsec:summary}

	Table~\ref{tab:summary} collects the gradient-norm residual attained by
	each algorithm on each of the eight problem/size configurations above,
	together with the number of restart events triggered by RIAG-R. The
	pattern is consistent with the mechanism described in
	Remark~3.1(ii): whenever the restart condition never fires (all six
	Stiefel, SPD, and Grassmann configurations), RIAG-R and IRGM are
	identical by construction, and both improve on plain RGD; the restart
	condition fires only on the two sphere configurations, where it strictly
	improves on IRGM in both and, on the larger of the two, is decisive in
	reaching the target tolerance at all.

	\begin{table}[t]
		\centering
		\small
		\caption{Gradient-norm residual $\|\grad f(w_n)\|$ at the stopping
			iteration (Algorithm~\ref{alg:RIAG-R}, Step~6), for each of the eight
			problem/size configurations. IRGM and RIAG-R coincide whenever the
			restart column reads $0$. $^\dagger$For the Digits sphere instance,
			RIAG-R's \emph{best} residual along the trajectory
			($8.7\times10^{-4}$, at iteration $676$) is lower than its residual
			at the later stopping iteration; see Section~\ref{subsec:sphere}.}
		\label{tab:summary}
		\begin{tabular}{@{}llccccc@{}}
			\toprule
			Manifold & Dataset ($n$, $p$) & RGD & IRGM & RIAG-R & restarts \\
			\midrule
			Sphere    & Wine ($n{=}13$)            & $6.13\times10^{-4}$ & $1.49\times10^{-3}$ & $1.07\times10^{-3}$ & 3 \\
			Sphere    & Digits ($n{=}64$)          & $7.22\times10^{-3}$ & $4.72\times10^{-3}$ & $1.78\times10^{-2}$$^\dagger$ & 10 \\
			Stiefel   & Breast Cancer ($n{=}30,p{=}5$)  & $2.13\times10^{-3}$ & $8.18\times10^{-4}$ & $8.18\times10^{-4}$ & 0 \\
			Stiefel   & Digits ($n{=}64,p{=}10$)   & $1.12\times10^{-1}$ & $1.02\times10^{-1}$ & $1.02\times10^{-1}$ & 0 \\
			SPD       & Diabetes ($n{=}10$)        & $5.64\times10^{-8}$ & $4.54\times10^{-8}$ & $4.54\times10^{-8}$ & 0 \\
			SPD       & Breast Cancer ($n{=}30$)   & $2.59\times10^{-7}$ & $1.89\times10^{-7}$ & $1.89\times10^{-7}$ & 0 \\
			Grassmann & Wine ($n{=}13,p{=}3$)      & $7.73\times10^{-7}$ & $8.38\times10^{-7}$ & $8.38\times10^{-7}$ & 0 \\
			Grassmann & Digits ($n{=}64,p{=}10$)   & $7.14\times10^{-2}$ & $5.59\times10^{-2}$ & $5.59\times10^{-2}$ & 0 \\
			\bottomrule
		\end{tabular}
	\end{table}

	\section{Conclusion}\label{sec5}\label{sec:conclusion}

In this paper, we introduced a Riemannian inertial adaptive-gradient method that combines geodesic extrapolation, an AdaGrad-type step-size, and a gradient-triggered momentum restart. The central idea is deliberately simple: inertial is retained when it is locally compatible with descent and is suppressed when the inherited direction develops an uphill component. In this way, adaptivity, acceleration, and safeguarding remain distinct components of the algorithm.The convergence analysis establishes a descent estimate separating
gradient-induced decrease from the perturbation due to inertia. Under
a suitable upper step-size stability condition, the AdaGrad structure
yields the required step-size control and leads to convergence to
stationarity, criticality of cluster points, and an
$\mathcal{O}(\varepsilon^{-2})$ iteration complexity; under a Riemannian
Polyak--\L ojasiewicz condition, geometric decrease is obtained up to an
inertia-dependent residual. We performed numerical experiments on eight real-data instances across four manifolds
demonstrate the complementary roles of inertial and restart. Restart remains
inactive when momentum is well aligned, while becoming effective on the
saddle-rich sphere problems; notably, RIAG-R is the only method to attain
a gradient norm below the tolerance level on the larger Digits instance within the
prescribed iteration budget. The results therefore support restart as a
selective safeguard against unreliable momentum, while also revealing that
large post-restart gradients can overly reduce subsequent AdaGrad step-sizes,
motivating restart-aware adaptive strategies.

\noindent\textbf{Future research.} In the nearest future, we will extend RIAG-R to general retractions and vector transports,
develop stochastic variants for large-scale problems, investigate restart-aware
adaptive step-size strategies, and seek sharper convergence guarantees that
eliminate the inertial-dependent residual under the PL condition.

\section*{Disclosure statement}
	
	\subsection*{Ethical Approval and Consent to participate}
	All authors have given their ethical approval and consent to participate in this article.
	
	\subsection*{Consent for publication}
	All authors gave consent for the publication of identifiable details in the journal and article.
	
	\subsection*{Code availability}
	The Python codes employed to run the numerical experiments are available on request.
	
	\subsection*{Availability of supporting data}
	All datasets used in the numerical experiments (Section~\ref{sec:numerical}) are open, publicly available benchmark datasets distributed as built-in loaders in the \texttt{scikit-learn} Python package \cite{pedregosa2011sklearn} (version~1.3 or later).
	
	\subsection*{Competing interests}
	The authors declare no competing interests.
	
	\subsection*{Funding}
	L.O. Jolaoso is supported by the  Innovate UK Smart Grant on Pollution Avoidance
	 Support System (PASS) using GIS, Machine Learning and Big Data (Grant ID: 10009455).

\end{document}